\documentclass{amsart}
\usepackage[utf8]{inputenc}

\usepackage{geometry}
 \usepackage[colorlinks=true, linkcolor=blue, citecolor=blue, urlcolor=magenta]{hyperref}

\usepackage{amsmath,amsthm,amssymb}
\usepackage{bbm}
\usepackage{graphicx}
\usepackage{subcaption}
\usepackage{todonotes}
\usepackage{cleveref}
\usepackage{color}
\usepackage[style=ext-numeric,sorting=nyt, url = false, doi = false, eprint = false, isbn = false,giveninits=true, articlein=false,uniquename=init]{biblatex}

\usepackage{cancel}
\numberwithin{equation}{section}

\newtheorem{theorem}{Theorem}[section]
\newtheorem*{theorem*}{Theorem}
\newtheorem{defn}[theorem]{Definition}
\newtheorem{proposition}[theorem]{Proposition}
\newtheorem{lemma}[theorem]{Lemma}
\newtheorem{cor}[theorem]{Corollary}
\newtheorem{remark}[theorem]{Remark}

\newcommand{\epm}{e^{\pm i\Theta (x)}}

\newcommand{\R}{\mathbb{R}}
\newcommand{\N}{\mathbb{N}}
\newcommand{\Z}{\mathbb{Z}}
\newcommand{\C}{\mathbb{C}}

\newcommand\numberthis{\addtocounter{equation}{1}\tag{\theequation}}

\title[Jacobi Half Line Dispersive]{Dispersive Decay Estimates for periodic Jacobi operators on the half-line}

\author{Amir Sagiv}
\address{Department of Mathematical Sciences, New Jersey Institute of Technology, Newark, NJ 07102, USA}
\email{amir.sagiv@njit.edu}
\author{Remy Kassem}
\address{Department of Applied Physics and Applied Mathematics, Columbia University, 500 W120 Street, New York, NY 10027, USA}
\email{rhk2130@columbia.edu}
\author{Michael I.\ Weinstein}
\address{Department of Applied Physics and Applied Mathematics and Department of Mathematics, Columbia University, 500 W120 Street, New York, NY 10027, USA}
\email{miw2103@columbia.edu}

\begin{document}
\maketitle
\begin{abstract}
We establish dispersive time-decay estimates for periodic Jacobi operators on the discrete half-line, $\N$. Specifically, we prove $t^{-1/2}$ decay in the weighted $\ell^\infty_{-1}$ norm for all such operators. For the global $\ell^1 \to \ell^\infty$ decay estimate, we show that $t^{-1/3}$ decay holds under a nondegeneracy condition on the discriminant. Alternatively, for any even period $q\geq2$, if the continuous spectrum consists of exactly $q$ disjoint intervals (bands), we obtain a $t^{-1/(q+1)}$ decay rate without any further assumptions.
\end{abstract}

\section{Introduction}
We study dispersive time-decay for the time-dependent Schrödinger equation on the discrete half-line, $\mathbb{N}=\{1,2,3,\ldots \}$, with a periodic Jacobi operator $J$. That is, we consider
\begin{equation*}
i\partial_t \psi(t) = J \psi(t) \, , \qquad \psi(0) = u \in \ell^2(\mathbb{N}) \, ,
\end{equation*}
where $J$ is a real, self-adjoint, nearest-neighbor difference operator 
\begin{subequations}\label{eq:JacobiDef}
\begin{equation}\label{eq:genJacobi}
    J = \begin{pmatrix}
        b_1 & a_1 & & &\\
        a_1 & b_2 & a_2 & &\\
        & a_2 & b_3 &a_3 &\\
        & & a_3 &\ddots  &\ddots \\
        & & & \ddots & 
    \end{pmatrix} \, .
\end{equation}
Assume $J$ has a minimal period, $q\geq 1$, such that for all $n\in \mathbb{N}$,
\begin{equation}\label{eq:perCoeff}
    a_{n+q}=a_n > 0 \, , \qquad b_{n+q}=b_n \in \R \, .
\end{equation}
\end{subequations}

Denoting the projection onto the continuous spectrum of $J$ by $P_{\rm c}$, our main results are operator-norm dispersive decay estimates for $e^{-itJ} P_{\rm c}$:
\begin{enumerate}
\item A local, weighted decay bound $\|e^{-itJ} P_{\rm c}\|_{\ell^1_1 \to \ell^\infty_{-1}} \lesssim t^{-1/2}$ for all periodic $J$ (Theorem \ref{thm:mainLocal}).
\item A global, unweighted decay bound $\|e^{-itJ} P_{\rm c}\|_{\ell^1 \to \ell^\infty} \lesssim t^{-1/3}$ under an algebraic condition (Theorem \ref{thm:mainGlobal}, part (1)).
\item A global, unweighted decay bound $\|e^{-itJ} P_{\rm c}\|_{\ell^1 \to \ell^\infty} \lesssim t^{-1/(q+1)}$ when $q$ is even and the spectrum consists of exactly $q$  disjoint intervals / bands (Theorem \ref{thm:mainGlobal}, part (2)).
\end{enumerate}

\subsection{Background - Dispersive decay estimates}

Dispersive decay estimates are a classical topic in analysis, focused on the quantitative study of how solutions to dispersive equations decay and spread over time. Generally, let $H$ be a self-adjoint operator on a Hilbert space $\mathcal{H}$, and denote the projection into its continuous spectrum by $P_{\rm c}$ \cite{cycon1987schrodinger, hall2013quantum}. Then, RAGE theorem guarantees qualitative dispersion of $e^{-itH}P_{\rm c}$ \cite{amrein1973characterization, cycon1987schrodinger, enss1978asymptotic, ruelle1969remark}. More informative quantitative time-decay bounds on $e^{-itH}P_{\rm c}$, however, are usually model-specific. Since the flow is unitary, i.e., $\|e^{-itH}\|_{\mathcal{H}}=1$, the spreading of the solution is equivalent to the decay of its maxima in discrete settings.
 There is an extensive literature concerning dispersive decay estimates in PDE models such as the Schr{\"odinger} \cite{jensen1979spectral, komech2010weighted, kopylova2014dispersion,schlag2007dispersive} and Dirac equations \cite{ERDOGAN2DMassive, erdougan2021massless, erdougan2018dispersive, erdougan2019dispersive,  kraisler2024dispersive}, to name just a few examples.   

The study of dispersion in {\em discrete} systems is more recent. For the one-dimensional discrete Schr{\"o}dinger operator $H$ on $\Z$ with a sufficiently rapidly decaying potential, one has $\|e^{-itH}P_{\rm c}\|_{\ell^1 \to \ell^{\infty}} \lesssim t^{-1/3}$, with faster decay rates for weighted $\ell^2$ norms   \cite{egorova2015dispersion, komech2006dispersive, pelinovsky2008spectral}. For a $q$-periodic potential on $\Z$, the best known estimate is $\|e^{-itH}P_{\rm c}\|_{\ell^1 \to \ell^{\infty}} \lesssim t^{-\min\{1/3, 1/(q+1)\}}$ \cite{mi2022dispersive}; see related estimates for certain quasi-periodic potentials in \cite{bambusi2020dispersive}. Decay estimates were also obtained for the discrete Laguerre operator, a non-bounded and non-periodic Jacobi operator \cite{koornwinder2018jacobi, kostenko2016dispersion}, and for Jacobi operators with eventually-constant coefficients \cite{egorova2016properties}. 

As noted, we are interested in the dynamics of $e^{-itJ}$ for $t>0$ for periodic Jacobi operators of the form~\eqref{eq:JacobiDef}. A relevant direction in the study of the dynamics of $e^{-itJ}$ concerns the question of ballistic transport, i.e., whether the expectation of the position operator grows linearly in time; see \cite{damanik2024ballistic} for a broad survey of this topic. Ballistic transport holds for all periodic and block-periodic Jacobi operators on $\Z$ \cite{damanik2015quantum}, a property that extends to higher dimensions \cite{fillman2021ballistic}.  Notwithstanding, the question of dispersive decay estimates remains open: while ballistic transport shows that the expectation of position grows linearly in time, it provides no information about dispersion.  {\em To the best of our knowledge, there has been no study of dispersive decay estimates for {\em any} Schr{\"o}dinger equation on the half line, $\N$.}

\subsection{Why Jacobi operators, why on the half line $\N$?}

Jacobi operators in one dimension (on $\N$ or $\Z$, periodic or not) are a key object in mathematics and physics. They are in one-to-one correspondence with families of orthogonal polynomials, as Jacobi operators define three-term recurrence relations for orthogonal polynomials, and vice versa \cite{lukic2022first, simon2010szego}. In numerical analysis, their finite truncations form Jacobi matrices that are at the heart of Gaussian quadrature algorithms, including the classical Golub–Welsch method \cite{golub1969calculation, olver2020fast, townsend2015race}. In integrable systems, Jacobi operators appear in the Lax pair of the Toda lattice models \cite{ablowitz1981solitons, teschl2000jacobi, toda2012theory}.

 The motivation most relevant to this work is the ubiquity of Jacobi operators as models of periodic and more generally of one-dimensional structured materials. Taking $a_n \equiv 1$ for all $n\in \N$, periodic Jacobi operators are simply discrete Schr{\"o}dinger operators with a periodic potential. Extending beyond that case, one of the simplest and most important examples is the Su-Schrieffer-Heeger (SSH) model \cite{su1979solitons} -- \eqref{eq:JacobiDef} where $b_n \equiv 0$ for all $n\in \N$ and $q=2$. The SSH model is a paradigmatic model in 1D topological wave phenomena \cite{asboth2016short, chung2023topological, freed2013twisted, shapiro2022continuum}. More broadly, Jacobi operators describe chains of sites with nearest-neighbor hopping, as has been realized in such diverse settings as condensed matter physics \cite{meier2016observation}, acoustics \cite{xue2022topological}, and photonics \cite{parto2018edge}, to name a few.

Why are we interested in the half line? Our motivation comes from the study of periodic materials, and in particular from edge modes and edge conductance: the transport of energy along edges and its localization near edges or domain walls. 

Motivated by modeling and experiments in photonic waveguides \cite{bellec2017non, jurgensen2021quantized} and radiated graphene \cite{merboldt2025observation, sagiv2022effective}, one is led to consider time-periodic perturbations of the form
\[
i\partial_t \psi = (J + \varepsilon J_1(t)) \psi,
\]
which model parametric forcing in so-called {\em Floquet materials}. In a PDE analog of this problem, it was shown that for sufficiently strong forcing, localized modes become metastable, eventually leaking energy into the continuous spectrum and decaying in time \cite{hameedi2023radiative}. This is an example of a broader phenomenon known as radiation damping, or ionization, which has been extensively studied in the context of the Schr{\"o}dinger and Dirac equations; see, e.g., \cite{borrelli2022complete, costin2008ionization, costin2018nonperturbative,costin2001resonance,   kirr2003metastable, miller2000metastability,  soffer-weinstein:99}. {\em Dispersive decay estimates are key to establishing such ionization phenomena.} As part of a broader program aimed at understanding discrete wave phenomena under parametric forcing, the results in this work provide an essential step. Looking ahead, we expect {\em linear} dispersive decay estimates to play a key role in understanding the asymptotic stability of {\em nonlinear} waves in discrete systems.

In terms of techniques, while the coefficients in \eqref{eq:JacobiDef} are periodic, $J$ is not translation invariant, since it is defined on $\N$ rather than $\Z$. In e.g.,  \cite{abdul2024slow, mi2022dispersive}, periodic Schr{\"o}dinger operators on $\Z$ are studied using the discrete Fourier transform, which reduces the problem to a parametric family of finite-dimensional Jacobi matrices and their spectrum. On $\N$, this avenue is not available to us.

Motivated by the analysis in \cite{ito2017resolvent} of the discrete Schrödinger operator on $\N$ via the sine transform, and by foundational theory of eigenfunction expansions \cite{berezanskii1968expansions}, we use the fact that $\ell^2(\N)$ is cyclic under $J$; that is,
$
\ell^2(\N) = \overline{\mathrm{Span}\{J^k \delta_1\}_{k \geq 0}}$, 
where $\delta_1 = (1, 0, 0, \ldots)$ is the standard unit vector. We can then find a sequence of orthogonal polynomials $(p_n(x))_{n=0}^{\infty}$ so that the propagator admits the representation
\[
\left[ e^{-itJ} u \right]_n = \int_{\sigma(J)} \langle p(x), u \rangle_{\ell^2(\N)} \, p_{n-1}(x) e^{-itx} \, d\mu(x),
\]
where \( p(x) = (p_0(x), p_1(x), \ldots) \) and $\mu$ is the spectral measure, see Sec.\ \ref{sec:genJac}. In our settings, the periodicity condition \eqref{eq:perCoeff} allows us to compute both $p_n$ and $\mu$ explicitly.


In a forthcoming work \cite{kassem2025dispersive}, we prove dispersive time-decay estimates for Hermitian Hamiltonians, where the wave function is defined on a dimer lattice. This studies an extension of Jacobi operators in the case $q=2$, where the off-diagonal entries may be complex, while preserving self-adjointness, $J^*=J$. The analysis makes use of a one-sided discrete Fourier–Laplace transform to compute the resolvent and then to construct  the spectral measure and the propagator via Stone’s formula.

As we were completing this manuscript, we became aware of a parallel and significant work being carried out by Damanik, Fillman, and Young \cite{DFYpreprint}, concerning global dispersive decay estimates for periodic Jacobi operators on $\Z$, in contrast to the setting considered here, on $\N$.

\subsection{Structure of the paper.} Section \ref{sec:jacobi} provides a summary of relevant results regarding periodic Jacobi operators. Section \ref{sec:mainRes} presents the precise statements of the main results, Theorems \ref{thm:mainLocal} and~\ref{thm:mainGlobal}. We then proceed to construct the propagator $e^{-itJ}P_c$ and derive a simplified oscillatory integral in Section \ref{sec:prop}, and provide its estimates in Section \ref{sec:estimates}, thus proving the main results. 

\subsection{Acknowledgments} This research was supported in part by National Science Foundation grants 
DMS-1908657 and DMS-1937254 (MIW),  Simons Foundation Math + X Investigator Award \#376319 (MIW), and the Binational Science Foundation Research Grant \#2022254 (MIW, AS). The author would like to thank J.\ Breuer for introducing him to the world of Jacobi operators, and would like to thank N.\ Levi, M.\ Luki{\'c}, T.\ Malinovitch, and S.\ Steinerberger for useful pointers and suggestions along the way.


\section{Review of Jacobi operators on the half-line}\label{sec:jacobi}
We begin by recalling essential properties of bounded Jacobi operators on $\ell^2 (\N)$ in Sec.\ \ref{sec:genJac}. This provides foundation for the specialized case of periodic coefficients, see Sec.\ \ref{sec:perExpo}. This introduction touches only on the most essential properties that will be used throughout this work, and is based on the extensive manuscripts of \cite{berezanskii1968expansions, lukic2022first, simon2010szego, teschl2000jacobi}. 

\subsection{Bounded Jacobi operators on the half line.}\label{sec:genJac} 
Consider \eqref{eq:genJacobi} where $a_n>0$ and all the coefficients are bounded (but not necessarily periodic). Then $J$ is a bounded self-adjoint operator on $\ell^{2}(\N)$, whose action on any $v:\N \to \R$ is defined by  
$$ \left[ Jv\right]_n = \left\{ \begin{array}{cc}
     b_1 v_1 + a_1 v_2 \, ,&  n=1  \, ,\\
     a_{n-1}v_{n-1} + b_n v_n + a_n v_{n+1} \, , & n\geq 2 \, .
\end{array} \right. $$
Thus, the (formal) eigenvalue equation 
$
Jv(x) = xv(x)$ for any $x\in \C$ can be rewritten as a recursion
$$\begin{pmatrix}
    v_{n+1}(x) \\ a_n v_n (x)
\end{pmatrix} = A(a_n,b_n,x) \begin{pmatrix}
    v_n (x)\\ a_{n-1}v_{n-1}(x)
\end{pmatrix} \, , $$ 
where
\begin{equation}\label{eq:A_def}
    A(a,b,x) \equiv \frac{1}{a}\begin{pmatrix}
        x-b & -1 \\ a^2 &0 
    \end{pmatrix} \, .
\end{equation}
When iterated, this relation yields the transfer matrix representation
\begin{equation}\label{eq:Tdef}
    \begin{pmatrix}
    v_{n+1} (x) \\ a_n v_n (x)
\end{pmatrix} = T_n (x)\begin{pmatrix}
    v_1 (x) \\ 0
\end{pmatrix} \, , \qquad T_n (x) \equiv  A(a_n,b_n,x)A(a_{n-1}b_{n-1},x) \cdots A(a_1,b_1,x) \, . 
\end{equation}
Loosely speaking, Jacobi operators can be ``diagonalized'' such that every point in the spectrum corresponds to a one-dimensional space of generalized eigenfunctions \cite{berezanskii1968expansions} -- a considerable simplification of the general spectral theorem. There exists a unique solution to the eigenvalue equation $$Jv(x)=xv(x) \, , \qquad v_1(x)=1 \, ,$$ satisfying (by \eqref{eq:Tdef}) 
            \begin{equation}\label{eq:pn_formula}
        v_{n+1} (x) = [T_{n} (x)]_{1,1} \, .
    \end{equation}
    For every $n\geq 0$, $p_n (x)\equiv v_{n+1}(x)$ is a real-valued polynomial of degree $n$ in the variable $x$. Furthermore, $(p_n)_{n=0}^{\infty}$ is the sequence of polynomials $p_n:\sigma (J)\to \R$ which are orthogonal with respect to the spectral measure 
    \begin{equation}\label{eq:dmudef}
        d\mu(x) \equiv \langle \delta_1, dE(x)\delta_1 \rangle_{\ell ^2} \, . 
        \end{equation}
    Here $dE$ is the projection-valued spectral measure associated with $J$, and $\delta_1= (1,0,\ldots,0,\ldots)$ is the standard unit vector. Finally, we have the following {\em generalized eigenfunction representation}~\cite{berezanskii1968expansions}: for every measurable $g:\sigma(J)\to \mathbb{C}$ and every $u\in \ell^2 (\mathbb{N})$,
    \begin{equation}\label{eq:eigenfunction}
        \left[ g(J)u\right]_n = \int\limits_{\sigma(J)} \langle p(x),u\rangle_{\ell ^2(\N)} ~ p_{n-1} (x) g(x) \, d\mu (x) \, , \qquad {\rm where}\quad \langle p (x), u \rangle_{\ell^2(\N)} \equiv  \sum\limits_{m\geq 1} p_{m-1} (x) u_m \, . 
    \end{equation}

\subsection{Periodic Jacobi Operators}\label{sec:perExpo}

Next, focus on {\em periodic} Jacobi operators on the half-line.  $J$ is said to be of period $q\in \N$ if $q$ is the smallest integer for which $a_{n+q}=a_n $ and $b_{n+q}=b_n$ for all $n\in \N$. We now show that the periodicity allows us to compute the eigenfunction expansion~\eqref{eq:eigenfunction} in closed form.

By applying the periodicity condition to \eqref{eq:A_def} and \eqref{eq:Tdef}, then
    for every $n=sq+r$ with $0\leq r <q$,
    \begin{equation}\label{eq:Tper}
        T_{n}(x) = T_r (x) \left( T_q (x) \right)^s \, ,
    \end{equation}
    where $T_0 (q) = {\rm Id}_{2\times 2}$, and $T_q(x)$ is called the {\em monodromy matrix}. Diagonalizing the monodromy matrix by $T_q (x)=P^{-1}(x)D(x)P(x)$,  then $(T_q (x)) ^s = P^{-1}(x) D^s (x) P(x)$ can be computed once for all $s\geq 1$, and therefore by \eqref{eq:pn_formula} we can compute $p_n (x)$ for all $n\geq 1$, see Sec.\ \ref{sec:prop}.
    
    The monodromy matrix, $T_q (x)$, fully characterizes the spectral measure. This will allow us also to write the eigenfunction expansion \eqref{eq:eigenfunction} explicitly. Denote  
        \begin{equation}\label{eq:DeltaDef}
              \Delta (x) \equiv {\rm Trace}\left(T_q (x)\right) \, , \qquad t_{i,j}(x) \equiv \left[ T_q (x) \right]_{i,j} \, , \qquad i,j\in \{1,2\} \, .
        \end{equation}

        \begin{theorem}[See in particular \cite{lukic2022first}, Theorem 10.63 and \cite{simon2010szego}, Theorem 5.4.2]\label{thm:DeltaProps}
        \mbox{}
        \begin{enumerate}
            \item $\Delta (x)$ is a polynomial of degree $q$.
            \item $J$ has only continuous spectrum and pure point spectrum. Its continuous spectrum satisfies $ \sigma_{c}(J) =  \Delta^{-1}([-2,2])$ and can be divided into   $q$ bands (intervals) \begin{equation}\label{eq:bandStructure}
\sigma_{\rm c}(J) = \bigcup_{j=1}^q I_j \, , \qquad I_j \equiv [\lambda_{2j-1}, \lambda_{2j}] \, ,
\end{equation}
where $\lambda_1 < \lambda_2 \leq \lambda_3 \cdots \leq \lambda_{2j-1}<\lambda_{2j} \leq \lambda _{2j+1} \cdots <\lambda _{2q}$, for each of which $\Delta ^2 (\lambda _i)=4$.
\item Each $I_j$ is a set on which $\Delta(x)$ traverses the entire interval $[-2,2]$ monotonically (either decreasing or increasing). 
\item Define $\Theta (x):\sigma_{\rm c}\to [-\pi, 0]$  to satisfy
\begin{equation}\label{eq:Thetadef}
\Delta  (x) = 2\cos \left( \Theta (x)\right) \, .
\end{equation}
Then $\Theta(x)$ traverses $[-\pi, 0]$ monotonically on each $I_j$.
\item $\Delta ' (x)$ has exactly $q-1$ simple zeroes, denoted by $\kappa _1 < \cdots <\kappa _{q-1}$, and either $\lambda_{2j}<\kappa_j<\lambda_{2j+1}$ (the {\em gapped case}) or $\lambda_{2j}=\kappa_j=\lambda_{2j+1}$ (the {\em gapless case}).
    \item The continuous part of the spectral measure $d\mu_c (x) = w(x) \,dx$ has the density
    \begin{equation}\label{eq:wx}
        w(x) = \frac{\sqrt{4-\Delta ^2 (x)}}{|t_{2,1}(x)|} \, .
    \end{equation}
    \item  $J$ has at most $q-1$ simple eigenvalues at points $x\in \R$ where $t_{2,1}(x)=0$ and $\left|t_{1,1}(x)\right|<~1$. All eigenvalues of $J$ lie in gaps of the continuous spectrum.
        \end{enumerate}

\end{theorem}
        \begin{remark}
It is common to define $\Delta (x) = 2\cos (q\Theta (x))$, in which case $\Theta$ is the well-known as the Marchenko-Ostrovski map; for brevity, we omit the details of that construction here, see \cite{korotyaev2007marchenko} and \cite[Sec.\ 10.10]{lukic2022first} for details. 
\end{remark}

\section{Main Results}\label{sec:mainRes}
Our main result is a weighted (local) $\ell^1_1 \to \ell^{\infty}_{-1}$ dispersive decay estimate. Define the weighted $\ell^p$ norms:  for any $u:\N \to \C$, any $\sigma\in \R$, and $p\geq 1$,
$$ \|u\|_{\ell^p_{\sigma}} \equiv \left[\sum\limits_{n=1}^{\infty} \left| n^{\sigma} u_n\right|^p \right]^{1/p} \, , \qquad \qquad \|u\|_{\ell^{\infty}_{\sigma}} \equiv \sup\limits_{n\geq 1} \left| n^{\sigma}u_n\right| \, .$$
These norms allow us to quantify the spatial localization of functions on $\N$. 

\begin{theorem}\label{thm:mainLocal}
    Let $J$ be a periodic Jacobi operator as defined in \eqref{eq:JacobiDef}. Let $P_{\rm c}$ be the $\ell^2(\N)$-projection onto the continuous spectrum of $J$. Then
        $$\|e^{-itJ}P_{\rm c}u\|_{\ell^{\infty}_{-1}} \leq C t^{-1/2}\|u\|_{\ell^1_1} \, .$$
\end{theorem}
To understand this result, note first that $u\in \ell^1_1(\N)$ is a stronger localization condition than absolute summability, $u\in \ell^1(\N)$. Second, decay in the $\ell^{\infty}_{-1}$ norm describes not the global decay of the solution's maximum, but rather how the maximum escapes from the origin over time.

Crucial to the analysis of Theorem \ref{thm:mainLocal} (Sec.\ \ref{sec:localPF}) is that even if some spectral gaps close, there is always a band edge: simply because the spectrum is bounded. Moreover, generalizing the work of Borg for Schr{\"o}dinger operators, Flaschka proved that if $\sigma_{\rm c}(J)$ consists of a single interval, then $J$ is a discrete Schr{\"o}dinger operator with a constant potential \cite{flaschka2005discrete}; see also \cite[Theorem 5.4.21]{simon2010szego}. Therefore, we can always assume that there exists a gap in the spectrum of a periodic Jacobi operator $J$ with a minimal period $q\geq 2$.

To introduce our $\ell^{\infty}$ (global) dispersive decay estimates, we introduce the parametrization of the continuous spectrum, $\sigma_{\rm c} (J)$, via the discriminant $\Delta (x)$ and its associated phase function $k$:
defining $\Theta (x)$  as in \eqref{eq:Thetadef}, the restriction of $\Theta$ to any individual band, $\Theta _j \equiv \Theta|_{x\in I_j}:I_j \to [-\pi,0]$, is strongly monotonic. Define $k_j:[-\pi,0]\to I_j$ to be its inverse, i.e., $k_j(\Theta (x))=x$ for every $x\in I_j$. 

\begin{theorem}\label{thm:mainGlobal}
    Let $J$ be a periodic Jacobi operator with minimal period $q\geq 1$, see \eqref{eq:JacobiDef}.
    \begin{enumerate}
    \item Suppose that for each $1\leq j \leq q$ and every $\varphi \in [-\pi,0]$, if $k''_j(\varphi)=0$ then $k'''_j(\varphi)\neq 0$.
    Then 
    \begin{equation}\label{eq:global13rate}
    \|e^{-itJ}P_{\rm c}u\|_{\ell^{\infty}} \leq Ct^{-1/3}\|u\|_{\ell^1} \, .
    \end{equation}
    \item     Let $q\geq 2$ be even, and suppose that $\sigma_{\rm c}(J)$, the continuous spectrum of $J$, consists of {\em exactly} $q$ disjoint intervals $I_1,\ldots , I_q$ (with no conditions on the phase functions $k_j$). Then
        \begin{equation}\label{eq:globalevenq1}
        \|e^{-itJ}P_{\rm c}u\|_{\ell^{\infty}} \leq Ct^{-1/(q+1)}\|u\|_{\ell^1} \, .
        \end{equation}

    \end{enumerate}
\end{theorem}

\begin{remark}
Consider the Discrete Laplacian acting on $\ell^2 (\N)$, i.e., \eqref{eq:JacobiDef} with $q=1$. Without loss of generality, we may take $a_n =1$ and $b_n = 0$ for all $n\geq 1$. In this case, $\Delta (x) = x$, and so $k(\varphi)=2\cos (\varphi)$. Hence, if $k''(\varphi)=0 $ then $k'''(\varphi)\neq 0$ trivially, and therefore a $t^{-1/3}$ decay rate is implied by Theorem \ref{thm:mainGlobal}.  Similarly, in the case of the SSH model ($q=2$, all $b_n=0$), a more elaborate and concrete computation of the propagator (see Sec.\ \ref{sec:prop}) yields a global decay rate of $t^{-1/3}$. Thus, both in terms of local and global decay rates, there is no distinction between the discrete Laplacian and the SSH model.

For the global estimate \eqref{eq:global13rate}, the assumption concerning zeros of $k'''$ is necessary to achieve a Laplacian-like dispersion rate. It remains open whether this assumption is always or generically true, and what is the correct dispersive decay estimate when it fails.
\end{remark}

\section{The Propagator}\label{sec:prop}
We begin by computing the eigenfunction expansion \eqref{eq:eigenfunction}, where the orthogonal polynomials $p_n$ (see Sec.\ \ref{sec:genJac}) can be explicitly constructed using the monodromy matrix relation \eqref{eq:Tper}. Then in Sec.\ \ref{sec:propBands} we use the band structure of $J$ (Theorem \ref{thm:DeltaProps}) and the precise structure of the spectral measure $\mu$, see \eqref{eq:wx}, to simplify the propagator further to a form more amenable to oscillatory integral analysis, see \eqref{eq:Ft_fin}. Finally, in Sec.\ \ref{sec:kprop} we study the resulting phase functions $k_j(\varphi)$ which determine the dispersive dynamics of \eqref{eq:Ft_fin}.

Given $\Theta(x)$ as defined in Theorem \ref{thm:DeltaProps}, define for every $\ell \geq 0$,
\begin{equation}\label{eq:_rmdef}
    \varrho_{\ell}(x) \equiv \frac{\sin(\ell \Theta (x))}{\sin(\Theta (x))} \,.
\end{equation}
\begin{proposition}\label{prop:dispEstNoMeasurewL}
Consider a periodic Jacobi operator as in \eqref{eq:JacobiDef}. Then there exists a polynomial of degree $\leq 2q$, denoted by $L$, such that for every $\sigma \geq 0$,
\begin{equation}\label{eq:dispEstNoMeasurewL}
         \|e^{-itJ}P_c u \|_{\ell ^{\infty}_{-\sigma}} \lesssim \sup\limits_{n,m\in \N} \left| (nm)^{-\sigma}\int\limits_{\sigma_{\rm c}(J)} e^{-itx} L(x) \varrho_{n}(x)\varrho _{m}(x)  \, d\mu_{\rm c} (x) \right| \cdot \|u\|_{\ell ^1_{\sigma}} \, ,
\end{equation}
where $\mu_{\rm c}$ is the continuous part of the spectral measure of $J$, see \eqref{eq:dmudef}.
\end{proposition}
\begin{proof}
Fix $x\in \sigma _c (J)$. We are interested in the transfer matrix $T_q (x)$, see \eqref{eq:Tdef}, whose entries we denote by $t_{i,j}(x)$ for $i,j \in \{1,2\}$. To compute the $s$-th power of the monodromy matrix, $T_q ^s (x)$, we first diagonalize $T_q (x)$ and find that, since ${\rm det}(T_q (x))=1$ and $\Delta(x) \in [-2,2]$ (see Theorem \ref{thm:DeltaProps}), the complex-conjugate eigenvalues of $T_q (x)$ for $x\in \sigma_{\rm c} (J)$ are of the form
$$\mu _{\pm}(x) = \frac{\Delta}{2} \pm i \frac{\sqrt{4-\Delta^2 (x)}}{2} \, ,$$
where both fractions are real. Note that $|\mu_{\pm}(x)|=1$, and since ${\rm Re}(\mu_{\pm}(x))= \Delta(x)/2$, we can define $\Theta (x)$ as in \eqref{eq:Thetadef}.
Let $x\in \sigma_{\rm c} (J)$. Writing $\mu_{\pm}(x)=e^{\pm i\Theta (x)}$ and using elementary linear algebra, the corresponding eigenvectors of $T_q (x)$, denoted by $\pi_{\pm}(x)\in \C^2$, are 
$$\pi_{\pm}(x) = \begin{pmatrix}
    t_{1,2}(x) \\ a_{\pm} (x)
\end{pmatrix} \, , \qquad a_{\pm}(x) \equiv  \epm
-t_{1,1}(x) \, .$$
Since $a_- (x) = \bar{a}_+ (x)$, we can drop the $+$ subscript, and the change of basis matrices are therefore 
$$P (x) =\begin{pmatrix}
    t_{1,2}(x) & t_{1,2}(x) \\ a (x) & \bar{a} (x)
\end{pmatrix} \, , \qquad P^{-1} (x) = \frac{1}{2it_{1,2} (x)\sin (\Theta(x))} \begin{pmatrix}
    \bar{a} (x) & -t_{1,2}(x) \\ -a (x) & t_{1,2}(x)
\end{pmatrix} \, , $$
whence, by direct computation, for every $s\geq 1$,\footnote{The computation of the entries is straightforward. In the (2,1) entry, it is useful to note that $\cos(\Theta(x)) = \Delta (x)/2= (t_{1,1}(x)+t_{2,2}(x))/2$, and that since ${\rm det}(T_q (x))=1$, then $t_{1,1}(x)t_{2,2}(x)=1-t_{1,2}(x)t_{2,1}(x)$.}
\begin{align*}
    T_{sq} (x) &= P(x)\begin{pmatrix}
        e^{is\Theta (x)} & 0 \\ 0 & e^{-is\Theta (x)}
    \end{pmatrix} P^{-1}(x) \\
    &=\begin{pmatrix}
        \varrho_{s-1}(x) + t_{1,1}(x)\varrho_s (x) & t_{1,2}(x) \varrho_s (x) \\ t_{2,1}(x)\varrho_s (x) & \varrho_{s+1}(x) - t_{1,1}\varrho_s (x)  
    \end{pmatrix} \, ,  \numberthis \label{eq:Tmq}
\end{align*}
where the polynomials $\varrho_{\ell}$ are given in \eqref{eq:_rmdef}.
\begin{remark}
 By the definition of the transfer matrices \eqref{eq:Tdef}, each entry of $T_{n}(x)$ is a polynomial of degree $\leq n$ in the variable $x$. This is not apparent from the form of $\varrho_{\ell} (x)$ in \eqref{eq:_rmdef}.   
\end{remark}
The $sq$-th orthogonal polynomial is therefore given by
\begin{align*}
    p_{sq}(x) &= [T_{sq}(x)]_{1,1}  \\ 
    &= \varrho_{s-1}(x) + t_{1,1}(x) \varrho_s (x) \, . \numberthis \label{eq:pmq}
\end{align*}
 Similarly, factoring any $n\in \N$ as $n=sq+r$ where the remainder term is $r\in \{ 1,\ldots, q-1\}$, we have that $$p_{n}(x) = \left[A_r(x) \cdots A_1 (x) T_{sq}\right]_{1,1} \, , \qquad n=sq+r \, . $$
Therefore, by \eqref{eq:A_def} and \eqref{eq:Tmq},
\begin{equation*}
p_{n}(x) = \sum\limits_{j=-1,0,1} a_{r, j}(x) \varrho_{s+j}(x) \, .
\end{equation*}
Here, each $a_{r, j}(x)$ is a polynomial of degree $\leq r+q$, which crucially do not depend on $s$.
We can therefore write, for every $n=s_1q+r_1$,\footnote{It is just a matter of convenience here to express the propagator at the $n+1$ position, and does not influence the estimates later.}
\begin{align*}
\left[ g(J)P_c u\right]_{n+1} &= \int\limits_{\sigma_{\rm c}(J)} g(x) \langle p(x), u \rangle p_{n} (x) d\mu_{\rm c} (x) \, , \\
&= \sum\limits_{m=0}^{\infty}\int\limits_{\sigma_{\rm c}(J)} g(x) p_{m}(x) p_{n} (x) u_{m +1} \, d\mu_{\rm c} (x)\\ 
[m=s_2q+r_2] \qquad \qquad &= \sum\limits_{m=0}^{\infty}  u_{m +1}~\sum_{\substack{j_1=-1,0,1 \\ j_2=-1,0,1}}~\int\limits_{\sigma_{\rm c}(J)} g(x) a_{r_1,j_1}(x)a_{r_2,j_2}(x)\varrho_{s_1+j_1}(x)\varrho _{s_2+j_2}(x) \, d\mu_{\rm c} (x)\, .
\end{align*}
Hence, by the triangle inequality, and taking $g(J)=\exp(-itJ)$, we have that, for every $\sigma \geq 0$,
\begin{align*}
     &\left\| e^{-itJ}P_cu\right\|_{\ell^{\infty}_{-\sigma}} \\    
     &\lesssim \sup\limits_{n,m,j_1,j_2} \left|(nm)^{-\sigma}\int\limits_{\sigma_{\rm c}(J)} e^{-itx} a_{r_1,j_1}(x)a_{r_2,j_2}(x)\varrho_{s_1+j_1}(x)\varrho _{s_2+j_2}(x)  \, d\mu_{\rm c} (x) \right| \cdot \|u\|_{\ell ^1_{\sigma}}  \, . \numberthis \label{eq:supGlobalaj}
\end{align*}
Since the $a_{r,j}(x)$ are polynomials of degree $\leq 2q$, their products $a_{r_1,j_1}(x)a_{r_2,j_2}(x)$ form a finite family of (smooth) polynomials whose degrees do not depend on $n$ or $m$. Thus, taking the supremum in \eqref{eq:supGlobalaj} over all such polynomials, and over $j_1, j_2 \in \{-1,0,1\}$, affects the estimates only up to a constant. Finally, in  \eqref{eq:dispEstNoMeasurewL} we simplified the notation by dropping the factorization $n=s_1q+r_1$ and $m=s_2q+r_2$, as it is no longer needed after this point.
\end{proof}

\subsection{Band structure and the spectral measure}\label{sec:propBands}
To incorporate the spectral measure \eqref{eq:wx} and the band structure of Theorem \ref{thm:DeltaProps} into the result of Proposition \ref{prop:dispEstNoMeasurewL}, first note that by the triangle inequality, Proposition \ref{prop:dispEstNoMeasurewL} yields for $\sigma\geq 0$,
    \begin{equation}\label{eq:estPreIj}
         \|e^{-itJ}P_c u \|_{\ell ^{\infty}_{-\sigma}} \lesssim \sup\limits_{\substack{n,m\in \N \\ 1\leq j \leq q }} \left| (nm)^{-\sigma} \int\limits_{I_j} e^{-itx} L(x) \varrho_{n}(x)\varrho _{m}(x)  \, d\mu_{\rm c} (x) \right| \cdot \|u\|_{\ell ^1_{\sigma}} 
 \, .
\end{equation}

As noted, the continuous part of the spectral measure is given by \eqref{eq:wx}. Furthermore, by \cite[Thm.\ 10.76]{lukic2022first}, while $t_{2,1}$ is not a polynomial of definite sign, it does not change sign within any single band $I_j$ (see \eqref{eq:bandStructure}). Therefore, we assume without loss of generality that $t_{1,2}>0$ on $I_j$ and omit the absolute value sign in \eqref{eq:wx}, which together with \eqref{eq:Thetadef} yields
\begin{equation}\label{eq:wxnoabs}
    d\mu_{\rm c} (x) = \frac{2}{t_{2,1}(x)}\sin(\Theta_j (x))  \, dx\, .
\end{equation}

Fix $1\leq j\leq q$. Although $\Theta(x)$ is not invertible, its restriction $\Theta_j:I_j\to [-\pi,0]$ is invertible by Theorem \ref{thm:DeltaProps}.
Substituting \eqref{eq:wxnoabs} into \eqref{eq:estPreIj} and changing variables $\varphi \equiv \Theta_j(x)$, with $dx= \frac{\Delta'(k_j(\varphi))}{-2\sin\varphi}\,d\varphi$, yields
\begin{align*}
         \|e^{-itJ}P_c u \|_{\ell ^{\infty}_{-\sigma}} &\lesssim \sup\limits_{\substack{n,m \geq 0\\ 1\leq j \leq q }} \left|F(t,n,m,j,\sigma) \right|\|u\|_{\ell ^1_{\sigma}} \, , \qquad {\rm where} \\
         F(t,n,m,j,\sigma) &\equiv  (nm)^{-\sigma} \int\limits_{I_j} e^{-itx} L(x) \varrho_{n}(x)\varrho _{m}(x)  \, d\mu_{\rm c} (x)  \\ 
         &=(nm)^{-\sigma} \int\limits_{I_j} e^{-itx} L(x) \frac{\sin (n\Theta_j(x))\sin(m\Theta _j (x))}{\sin ^2(\Theta_j (x))} \frac{2}{t_{2,1}(x)} \sin (\Theta _j (x)) \, dx \\
         &= 4(nm)^{-\sigma} \int\limits_{-\pi}^0 e^{-itk_j(\varphi)} \sin(m\varphi) \sin(n\varphi) Q(k_j(\varphi)) \, d\varphi\, , \numberthis\label{eq:Ft_fin}
\end{align*}
and
\begin{equation}\label{eq:Qdef}
    Q(y) \equiv \frac{\Delta' (y) L(y)}{t_{2,1}(y)} \, ,
\end{equation}
is smooth (since $t_{2,1}$ does not vanish on $I_j$  \cite[Thm.\ 10.76]{lukic2022first}), and for all $\varphi\in [-\pi,0]$,
\begin{equation*}
k_j (\varphi)\equiv \Delta ^{-1}_j \left( 2\cos (\varphi)\right) \, .
\end{equation*}
In other words, $(k_j \circ\Theta _j) (x)=x$ on $I_j$.

\subsection{The phase functions $k_j(\varphi)$}\label{sec:kprop}

In Section \ref{sec:estimates}, we will use standard oscillatory integral techniques to bound $F(t,n,m,j,\sigma)$.
To prepare for this, we first analyze the phase functions $k_j$.
Differentiating $\Delta_j (k_j(\varphi))=2\cos (\varphi)$ repeatedly yields
\begin{align}
    \label{eq:deltak1der}
    \Delta _j ' (k_j (\varphi))k'_j (\varphi) &= -2\sin (\varphi) \, ,
\\
\label{eq:deltak2der}
    \Delta _j ''(k_j(\varphi)) [k_j '(\varphi)]^2 + \Delta_j '(k_j(\varphi))k''_j(\varphi) &= -2\cos (\varphi) \,,
\\
\label{eq:deltak3der}
    \Delta _j '''(k_j (\varphi)) [ k_j '(\varphi))]^3 + 3\Delta ''(k_j(\varphi))k'_j (\varphi)k''_j(\varphi) + \Delta_j '(k_j(\varphi))k'''_j (\varphi) &= 2\sin (\varphi ) \, .    
\end{align}
The key distinction to our analysis is between bands $I_j$ (see \eqref{eq:bandStructure}) whose endpoints are isolated by a spectral gap and those that are not. We make this precise:
\begin{defn}\label{def:gapped} 
    In the notations of Theorem \ref{thm:DeltaProps}, an endpoint $\lambda_i$ of a spectral band $I_j$ is \underline{gapped} if $\lambda_i \not\in I_{j-1}, I_{j+1}$.
    \end{defn}
The following corollary is an immediate consequence of Theorem \ref{thm:DeltaProps}:
\begin{cor}\label{cor:gappedDprime}
An endpoint $\lambda_i$ is gapped iff $\Delta' (\lambda_i)\neq 0$.
\end{cor}
In what follows, we treat the gapped and ungapped cases separately in  Lemmas \ref{lem:kgappedLemma} and \ref{lem:knogap}, respectively, followed by a general analysis regarding the first $q$ derivatives of $k_j$, Lemma \ref{lem:kevengap}.
\begin{lemma}\label{lem:kgappedLemma}
    Suppose $I_j$ is gapped on both of its endpoints, in the sense of Definition \ref{def:gapped}, for some $1\leq j \leq q$. Then $\Delta_j \equiv \Delta|_{I_j} $ is strictly monotonic on $I_j$, and 
    \begin{enumerate}
        \item $k_j'(\varphi_*)=0$ only for $\varphi_* = -\pi, 0$, and $k_j''(\varphi_*)\neq 0$.
        \item There exist  finitely many roots of $k_j''(\varphi)=0$ and finitely many roots of $k_j'''(\varphi)=0$.
    \end{enumerate}
\end{lemma}

\begin{proof}
Rewrite \eqref{eq:deltak1der} as $    k'_j(\varphi) = -2 \sin(\varphi)\left[\Delta_j ' \left( k_j(\varphi)\right)\right]^{-1} $. By Corollary \ref{cor:gappedDprime}, $\Delta_j'(k_j(\varphi_*))\neq 0$, and therefore $k'_j(\varphi)$ is continuous on $[-\pi,0]$. Thus, $k'_j(\varphi_*)=0$ iff $\varphi_* = -\pi,0$. 
Next, substituting $k'_j (\varphi_*)=0$ into \eqref{eq:deltak2der} yields
$$k_j''(-\pi)= \frac{4}{\Delta_j'(k_j(-\pi))} \, ,  \qquad k_j''(0)= \frac{-4}{\Delta_j'(k_j(0))} \, , $$
and so $k_j''(0)k_j''(-\pi)< 0$. Moreover, since $\Delta'$ does not change sign on $I_j$ (Theorem \ref{thm:DeltaProps}), and since  $k''_j$ is continuous, then $k_j''$ must vanish at some point in $(-\pi, 0)$. 

To see that $k_j'' (\varphi)=0$ has only finitely many roots, let $\varphi$ be such a point, and let $x = k_j (\varphi)\in I_j$. Then substituting $k''_j(\Theta _j(x))=0$ into \eqref{eq:deltak2der} yields 
$$ 0= \Delta (x)\left[\Delta ' (x)\right]^2 + \left(4-\Delta^2 (x)\right) \Delta '' (x) \, .$$
Since the right hand side is a polynomial, it can only have finitely many roots. Similarly, by substituting $k_j '''(\varphi)=0$ and the expressions for $k_j'$ and $k_j'''$ into \eqref{eq:deltak3der}, we find that the corresponding $x$ values satisfy a polynomial relation in $\Delta$, its derivatives, and $\sin (\Theta (x))$, and therefore can only have finitely many roots as well.
\end{proof}
\begin{lemma}\label{lem:knogap}
     Suppose $I_j$ has a non-gapped endpoint. Then $k_j(\varphi)=0$ only if $\varphi = 0,-\pi$ is a {\em gapped} endpoint, and $k_j\in C^2 ([-\pi, 0])$.
    \end{lemma}
    \begin{proof}
           Let us show that $k'_j(\tilde{\varphi})\neq 0$ if $k_j(\tilde{\varphi})$ is an {\em ungapped} endpoint. Without loss of generality, let $\tilde{\varphi}=0$.  Evaluating \eqref{eq:deltak2der} and noting that in the ungapped case $\Delta'(k_j (0))=0$ (Corollary \ref{cor:gappedDprime}), 
    $$\Delta ''(k_j(0))[k_j'(0)]^2 = -2 \, .$$
    By Theorem \ref{thm:DeltaProps}, all zeroes of $\Delta'$ are simple, and so $\Delta '' (k_j(0))\neq 0$. Therefore, $k'_j(0)\neq 0$. That $k_j'\neq 0$ elsewhere follow from our proof of Lemma \ref{lem:kgappedLemma}.
    
    To see that $k_j\in C^2 ((-\pi ,0))$, one can solve \eqref{eq:deltak2der} for $k_j ''(\varphi)$, and note that since $\Delta '\neq 0$ on the interior of each $I_j$, $k_j''(\varphi)$ is indeed continuous and well-defined. Hence all that is left to show for item (2) is that $k_j''(\tilde{\varphi})$ is defined at the endpoint (and therefore by \eqref{eq:deltak1der} continuous there). Substitute $\varphi = 0$ into \eqref{eq:deltak3der}, and since $\Delta ' _j(k_j(0))=0$ and $k'_j(0)\neq 0$, we obtain
    \begin{equation}\label{eq:kjppp0}
    k_j''(0)=\frac{-[k_j'(0)]^3 \Delta_j '''(k_j (0))}{3\Delta_j ''(k_j (0))} \, .
    \end{equation}
    Since the zeroes of $\Delta'$ are simple, $\Delta_j ''(k_j (0))\neq 0$, and so $k_j'' (0)$ is well defined, and so $k_j\in C^2([-\pi,0])$.  
\end{proof}

Finally, an alternative analysis for the case where $q\geq 2$ is even allows us to obtain a slower decay rate of $t^{-1/(q+1)}$ {\em without} a hypothesis on $k_j'''$ (Theorem \ref{thm:mainGlobal}, part (2)). To do so, we establish a lower bound on the first $q$ derivatives of $k$, inspired by \cite[Proposition 2.1]{mi2022dispersive}:
\begin{lemma}\label{lem:kevengap}
    Suppose $q\geq 2$ is even, and all bands are gapped.  Then there exists $c>0$ such that for all $\varphi\in[-\pi,0]$,
    $$ \sum \limits_{\ell =2}^{q} |k_j ^{(\ell)}(\varphi)| > c> 0 \, .$$
\end{lemma}
\begin{proof}
    Since the sum on the left hand side is continuous in $\varphi$, it is sufficient to show that it does not vanish on the compact set $[-\pi,0]$. Assume,  by contradiction, that $k''(\varphi_0)=\cdots =k^{(q)}(\varphi_0)=0$ for some point $\varphi_0\in [-\pi,0]$. Then by differentiating the relation $\Delta (k_j(\varphi))=2\cos (\varphi)$ exactly $q$ times, and each time setting $k_j^{(\ell)}(\varphi_0)=0$ for $2\leq \ell \leq q$, we get (for $q $ even)
    $$ \Delta _j^{(q)}(k(\varphi_0))[k'_j(\varphi_0)]^q = (-1)^{\frac{q}{2}} 2\cos (\varphi_0) \, .$$
    By definition of $\Delta$, see \eqref{eq:DeltaDef}, we have that $\Delta ^{(q)}(\varphi)\equiv q!/(a_1\cdot a_q)$ for all $\varphi$. Differentiating once again with respect to $\varphi$, we obtain
    $$0= (-1)^{\frac{q}{2}+1}2\sin (\varphi_0) \, .$$
    From the last equation, we deduce that $\varphi_0 = 0$ or $-\pi$. But that is a contradiction, since  by Lemma \ref{lem:kgappedLemma}, in the gapped case, $k''(0)\neq 0$ and $k''(-\pi)\neq 0$. Hence, $k^{(\ell)}_j (\varphi_0)=0$ for  for some $2\leq \ell \leq q$.
\end{proof}
Note that in ungapped intervals, by \eqref{eq:kjppp0}, $k''(0)$  vanishes if $\Delta'''$ vanishes at the respective endpoint of $I_j$. This  happens when $q=2$, since $\Delta$ is quadratic.\footnote{But this is not a very interesting case: the ungapped case in $q=2$ is simply the Laplacian, by Borg's Theorem~\cite{flaschka2005discrete}.}

\section{Oscillatory Integrals Estimates}\label{sec:estimates}

 We now recall the classical Van Der Corput lemma:
    \begin{lemma}[Van der Corput \cite{stein1993harmonic}]\label{lem:vandercorput}
    Let $\lambda$ be a smooth function and $f$ a smooth, compactly supported function. Suppose there exists $s\geq 2$ and $\lambda_0>0$, such that $\vert\lambda^{(s)}(z)\vert \geq \lambda_0$. Then there exists a constant, $c_s$, depending only on $s$ such that
    \begin{align}
        \left\vert\int_{\R} f(z)e^{i\lambda(z)} \, dz\right\vert \leq c_s \lambda_0^{-1/s}\| f'\|_{L^1}\ .
    \end{align}
\end{lemma}
When a lower bound on the first derivative is available, we use a variation of a well-known result, see e.g., \cite[Section 5.2]{miller2006applied}, proven by integration by parts.
\begin{lemma}[Oscillatory integrals with no stationary phase]\label{lem:1derbd}
    Let $\lambda\in W^{2,1} ([a,b])\cap C^1([a,b])$ and $f\in C^1 ([a,b])$. If there exists $\lambda_0>0$, such that $\vert\lambda'(z)\vert \geq \lambda_0$, then
    \begin{align}
        \left\vert\int\limits_{a}^b f(z)e^{i\lambda(z)} \, dz\right\vert \leq \lambda_0 ^{-1} \left[ 2\|f\|_{L^{\infty}}+ (b-a)\|f'\|_{L^{\infty}} +\lambda_0^{-1} \|f\|_{L^{\infty}} \cdot \|\lambda''\|_{L^1} \right] \, .
    \end{align}
\end{lemma}
\subsection{Proof of the local decay estimate, Theorem \ref{thm:mainLocal} }\label{sec:localPF}
Our strategy is to apply stationary phase analysis (Lemmas \ref{lem:vandercorput} and \ref{lem:1derbd}) to each interval $I_j$ separately, proving the following:\footnote{While we prove upper bounds, sharpness via asymptotic expansions is standard; see, e.g., \cite{kraisler2025time}. }
\begin{proposition}\label{prop:estLocal}
    Consider $F(t,n,m,j,1)$ as defined in \eqref{eq:Ft_fin} for fixed $n,m,$ and $j$. Then
    \begin{equation}\label{eq:Fgapped_local}
        \left| F(t,n,m,j,1)\right| \leq \begin{cases}
ct^{-1/2}, &\text{if $I_{j}$ has at least one gapped edge} \, ,\\[2pt]
ct^{-1},   &\text{otherwise} \, ,
\end{cases}
    \end{equation}
    for $c>0$ independent of $n,m,$ and $t$, and where we use ``gapped'' in the sense of Definition \ref{def:gapped}.
\end{proposition}
    \begin{proof}[Proof of Theorem \ref{thm:mainLocal} given Proposition \ref{prop:estLocal}]
    Since $\sigma_{\rm c}(J)$ is bounded, at least $I_1$ and $I_q$ have a gapped endpoint. Thus, for $t\gg 1$ sufficiently large, the $t^{-1/2}$ decay associated with that gapped endpoint is dominant. Hence,
     by \eqref{eq:Ft_fin}, the result follows.
     \end{proof}   

\begin{proof}[Proof of Proposition \ref{prop:estLocal}]
 Let us first treat the case where both endpoints are gapped.
The strategy is to decompose the integral in the definition of $F$, \eqref{eq:Ft_fin}, into a part where there are points of stationary phase, $F^{(s)}$, and one where there are none, $F^{(o)}$. Let $\chi\in C_c ^{\infty}(\R)$ be a smooth characteristic function defined by
\begin{equation}\label{eq:cutoff}
    \chi(z) = \begin{cases}
    1 & \textrm{if}\ |z| < 1 \\
    0 & \textrm{if}\ |z| > 2  \, .
    \end{cases}
\end{equation}
By Lemma \ref{lem:kgappedLemma}, there exist finitely many points $-\pi <\varphi_1< \ldots <\varphi_n<0$ for which $k''(\varphi_{\ell})=0$ for $1\leq \ell \leq n$. Let
\begin{equation}\label{eq:chij_gapped_loc}
\chi _j(\varphi) \equiv \chi\left(\frac{\varphi}{\delta}\right)+\chi\left(\frac{\varphi+\pi}{\delta}\right) \, , \qquad \delta \equiv \frac{1}{5}\min(|-\pi-\varphi_1|, |\varphi_n|) \, .
\end{equation}
Then, by its definition in \eqref{eq:Ft_fin},
\begin{align*}
    F(t,n,m,j,1) &= F^{(s)}_1(t,n,m,j) + F^{(o)}_1(t,n,m,j) \, , \qquad {\rm where} \numberthis \label{eq:F1_2_FsFo}\\
    F^{(s)}_1 (t,n,m,j) &\equiv   4(nm)^{-1} \int\limits_{-\pi}^0 \chi _j (\varphi)e^{-itk_j(\varphi)} \sin(m\varphi) \sin(n\varphi) Q(k_j(\varphi)) \, d\varphi\, , \numberthis \label{eq:Fsgap} \\
    F^{(o)}_1 (t,n,m,j) &\equiv   4(nm)^{-1} \int\limits_{-\pi}^0 \left(1-\chi _j (\varphi)\right)e^{-itk_j(\varphi)} \sin(m\varphi) \sin(n\varphi) Q(k_j(\varphi)) \, d\varphi\, . \numberthis \label{eq:Fogap}
\end{align*}
Our choice of $\delta$ above is such that the support of the integrand in $F^{(s)}_1$, see \eqref{eq:Fsgap}, includes no points where $k''(\varphi)=0$, and thus by continuity there exists a lower bound $k_0>0$ such that $k''(\varphi)>k_0>0$ on the support of the integrand. We can therefore apply Van der Corput lemma (Lemma \ref{lem:vandercorput}) 
\begin{align*}
    \left| F^{(s)}_1(t,n,m,j)\right| &\leq (nm)^{-1}\frac{c_k}{(tk_0)^{1/2}}\left\| \frac{d}{d\varphi} \Big[\chi_j (\varphi)\sin(m\varphi)\sin (n\varphi) Q(k_j(\varphi))\Big]\right\|_{L^1_{\varphi}} \\
    &\leq (nm)^{-1}\frac{c_k}{(tk_0)^{1/2}}\left\|\frac{d}{d\varphi} \left[\chi_j (\varphi)Q(k_j(\varphi))\right]\right\|_{L^{\infty}_{\varphi}} \left\| \sin(m\varphi)\sin (n\varphi) \right\|_{L^1_{\varphi}} \\
    &+(nm)^{-1}\frac{c_k}{(tk_0)^{1/2}} \left\|\frac{d}{d\varphi} \left[ \sin(n\varphi)\sin(m\varphi)\right]\right\|_{L^{\infty}_{\varphi}} \left\| \chi_j\cdot (Q\circ k)\right\|_{L^1} &\\
   &\equiv (nm)^{-1}\frac{c_k}{(tk_0)^{1/2}}\left[ {\rm I} + {\rm II}\right]  \, . \numberthis \label{eq:I_II_Fsloc}
\end{align*}
To bound the two terms in \eqref{eq:I_II_Fsloc}, we first prove a short technical lemma.
\begin{lemma}\label{lem:Qprop}
    $\psi (\varphi)\equiv Q\circ k (\varphi) \in C^1 ([-\varphi,0])$.
\end{lemma}
\begin{proof}
    By \eqref{eq:Qdef}, we see that $Q$ is smooth on any $I_j$:  it is a rational function, and the denominator $t_{2,1}$ does not vanish in the spectrum \cite[Thm.\ 10.76]{lukic2022first}. By Lemmas \ref{lem:kgappedLemma} and \ref{lem:knogap}, $k\in C^1([-\pi,0])$. Therefore, by the chain rule, $\psi \in C^1 ([-\pi ,0])$.
\end{proof}

By Lemma \ref{lem:Qprop} and the definition of $\chi_j$, \eqref{eq:chij_gapped_loc}, 
$[\chi_j (\varphi) Q\left(k_j(\varphi)\right)]\in C^1([-\pi,0])$.  Hence
$${\rm I} \lesssim \|\sin (m\varphi)\sin(n\varphi)\|_{L^1} \lesssim 1 \, .$$
Similarly in ${\rm II}$, Lemma \ref{lem:Qprop} implies that $\|\chi_j \cdot (Q\circ k)\|_{L^1} \lesssim 1 $ uniformly in $n,m$. Using the elementary identity $[\sin (\ell \varphi)]' = \ell \cos (\ell \varphi)$, we have that 
$${\rm II} \lesssim \max (|n+m|,|n-m|) \, ,$$
Subtituting the bounds for the two terms into \eqref{eq:I_II_Fsloc}, we proved that there exists $C>0$ independent of $n,m,$ and $t$ such that
\begin{equation}\label{eq:Fs1_bd_gapped}
  \left|F_1^{(s)}(t,n,m,j) \right|\leq Ct^{-1/2} \, .
\end{equation}

We now bound $F^{(o)}_1$ from above, see \eqref{eq:Fogap}. Since $0$ and $-\pi$ are not in the support of the integrand, there is a lower bound on the first derivative. We can thus apply Lemma \ref{lem:1derbd} as follows:
in the notations of the lemma, let $f(\varphi)=(1-\chi_j)\sin(m\varphi)\sin (n\varphi)Q(k_j(\varphi))$, and $\lambda (\varphi)=-k(\varphi)t$. Then by the same arguments as in the bound above for $F^{(s)}$, and by Lemma \ref{lem:Qprop}, there exists a universal constant $c>0$ such that $\|f\|_{L^{\infty}}\leq c  $ and $ \|f'\|_{L^{\infty}} \leq c(m+n) $. Next, by Lemma \ref{lem:kgappedLemma}, $\|\lambda''(\varphi)\|_{L^1}=-\|k''(\varphi)\|_{L^1} t < \infty$. Thus, applying the non-stationary estimate in Lemma \ref{lem:1derbd}, we obtain in the gapped case, for some $C>0$
\begin{equation}\label{eq:Fo1_bd_gapped}
  \left|F_1^{(o)}(t,n,m,j) \right|\leq Ct^{-1} \, .
\end{equation}
By \eqref{eq:F1_2_FsFo}, the partition of $F(t,n,m,1)$ into a stationary and a non-stationary part, the upper bounds \eqref{eq:Fs1_bd_gapped} and \eqref{eq:Fo1_bd_gapped} imply \eqref{eq:Fgapped_local}, for the case where both endpoints of $I_j$ are gapped.

Suppose now that $I_j$ is not gapped at either one of its ends (or both). The same proof follows with the natural change to $\chi_j$: if e.g., $I_j \cap I_{j-1}=\emptyset$ but $I_j\cap I_{j+1}=\lambda_{2j}$, then choose $\chi_j (\varphi)\equiv \chi (\varphi/\delta)$, where $\chi$ is defined in \eqref{eq:cutoff}. In applying Lemma \ref{lem:1derbd}, we see that $k''\in L^1$ by  Lemma \ref{lem:knogap}. Finally, if both endpoints are ungapped in \eqref{eq:Fgapped_local} then we set  $F=F^{(0)}$, and obtain a $t^{-1}$ rate as in \eqref{eq:Fo1_bd_gapped}.
\end{proof}

\subsection{Proof of the global estimates, Theorem \ref{thm:mainGlobal}}
We now prove the $t^{-1/3}$ bound in \eqref{eq:global13rate}. The even $q$ case, \eqref{eq:globalevenq1}, is addressed at the end of this section.
Analogously to the local estimate (Section \ref{sec:localPF}), we decompose the integral into neighborhoods of stationary points and ones where there is a uniform lower bound on $|k_j''|$. Let $T_i\equiv \{\varphi \in [-\pi,0] ~~ | k^{(i)}(\varphi)=0 \}$ for $i=2,3$. Since, by $T_2$ and $T_3$ are disjoint by hypothesis and finite by Lemma \ref{lem:kgappedLemma}, let 
\begin{equation}\label{eq:etadef}
X _j (\varphi) \equiv \sum\limits_{\alpha \in T_2} \chi \left(\frac{\varphi-\alpha}{\eta}\right)  \, , \quad \eta \equiv \frac15 \min \left\{ |\alpha - \beta| ~~ {\Large|}~~ \alpha \in T_2  ~ \wedge ~  \beta \in T_3\cup T_2\cup \{-\pi,0\} ~ \wedge ~ \alpha \neq \beta ~~ \right\}>0 \, .
\end{equation}

We write $F(t,n,m,j,0) = F_0^{(s)} + F_0^{(o)}$ with
\begin{align}
F^{(s)}_0 (t,n,m,j) &\equiv \int\limits_{-\pi}^0 X_j(\varphi) e^{-itk_j(\varphi)} \sin(m\varphi) \sin(n\varphi) Q(k_j(\varphi)) \, d\varphi \, , \label{eq:Fs0} \\
F^{(o)}_0 (t,n,m,j) &\equiv \int\limits_{-\pi}^0 (1 - X_j(\varphi)) e^{-itk_j(\varphi)} \sin(m\varphi) \sin(n\varphi) Q(k_j(\varphi)) \, d\varphi \label{eq:Fo0} \, .
\end{align}
  By the elementary identity
\begin{equation}\label{eq:sinmn}
    \sin(m\varphi) \sin(n\varphi) = -\frac{1}{4}\left(e^{i(n+m)}+e^{-i(n+m)}-e^{i(n-m)}-e^{i(m-n)} \right) \, ,
\end{equation}
the integral for $F^{(s)}_0$, see \eqref{eq:Fs0}, decomposes into four integrals corresponding to each one of the summands. Choose one point $\varphi_p\in (-\pi,0)$ for which $k''(\varphi_p)=0$, as guaranteed by Lemma \ref{lem:kgappedLemma}, and choose one of the four complex exponentials in \eqref{eq:sinmn}. The corresponding integral is, up to a constant
\begin{equation}\label{eq:sin2exp}
\int\limits_{-\pi}^0  \exp\left[-it\left( k_j(\varphi)- \frac{\ell}{t}\right)\right] X_j(\varphi) Q(k_j(\varphi)) \, d\varphi \, ,
\end{equation}
where $\ell\in \Z$ is either $\pm(n+m)$ or $\pm(n-m)$.
Defining $$\alpha(t,\ell,\varphi)\equiv k_j(\varphi)-\ell /t \, ,$$ there exists a sequence of $t_i \to \infty$ and $\ell_i\to \infty$ such that 
$$\partial_{\varphi} \alpha (t_i, \ell_i,\varphi_p)=\partial_{\varphi \varphi}\alpha (t_i,\ell_i,\varphi_p) = 0 \, , \qquad {\rm but}\qquad   \partial _{\varphi \varphi \varphi }\alpha (t_i,\ell_i, \varphi_p) 
 = k'''(\varphi_p) \neq 0 \,,$$
where the last inequality is due to our hypothesis that $T_2 \cap T_3 =\emptyset$.
Thus, by the choice of the scale $\eta>0$ in \eqref{eq:etadef}, we can apply Van der Corput lemma (Lemma \ref{lem:vandercorput}) with $s=3$ to bound \eqref{eq:sin2exp} when $t=t_i, \ell = \ell_i$ from above, yielding the desired $t^{-1/3}$ bound. For any other term of the form $\{F^{(s)}_0 (t_i,n,m,j)\}_{n,m}$, and all times $t\neq t_i$, there is a lower bound on the {\em first} derivative of $\alpha$, and therefore Lemma \ref{lem:1derbd} is applicable, yielding a faster $t^{-1}$ decay rate. We have thus proved that, for some $C>0$ independent of $n,m,$ and $t$,
\begin{equation}\label{eq:Fs0_bd_gapped}
  \left|F_0^{(s)}(t,n,m,j) \right|\leq Ct^{-1/3} \, .
\end{equation}
To bound $F^{(o)}$ from above, note that $\varphi_* =-\pi, 0$ are in the support of the integrand, due to the choice of the scale $\eta>0$ in \eqref{eq:etadef}, and $T_2$ is disjoint from the support of the integrand. Thus, $|k''(\varphi)|>c>0$ on the domain of integration, and applying Van der Corput Lemma \ref{lem:vandercorput} yields
\begin{equation}\label{eq:Fo0_bd_gapped}
  \left|F_0^{(o)}(t,n,m,j) \right|\leq Ct^{-1/2} \, .
\end{equation}
Applying \eqref{eq:Fs0_bd_gapped} and \eqref{eq:Fo0_bd_gapped} to the decomposition of $F(t,n,m,0)$ yields the desired overall $t^{-1/3}$ decay rate, concluding the proof of \eqref{eq:global13rate} for a gapped interval.

When $I_j$ is ungapped (Definition \ref{def:gapped}), the only change is that $k''_j$ might not vanish at all. In that case, $F^{(s)}_0=0$, and so the overall rate of decay for that interval will be $t^{-1/2}$.

{\bf Proof of \eqref{eq:globalevenq1}.} In the case where $q\geq 2$ is even, the proof of \eqref{eq:globalevenq1} follows identically. The only difference is as follows: by Lemma \ref{lem:kevengap}, we know that for each $\varphi \in [-\pi,0]$, there exists $2\leq \ell \leq q$ such that $|k^{(\ell)} (\varphi)|>c>0$. Therefore, we can partition the domain of integration (using the cutoff function \eqref{eq:cutoff}, as before) to points where $|k^{(q)}(\varphi)|>c>0$, and segments where some derivative of order $2\leq \ell < q$ is bounded from below. By similar arguments to the proof of \eqref{eq:global13rate}, those points for which only a nonzero lower bound on $k^{(q)}$ is available would yield a $t^{-1/(q+1)}$ decay bound, which would dominate the estimate.

\printbibliography

\end{document}